\documentclass[10pt]{amsart}
\usepackage{amssymb}

\usepackage{enumitem}

\newtheorem{thm}{Theorem}[section]
\newtheorem{prop}[thm]{Proposition}
\newtheorem{lem}[thm]{Lemma}

\theoremstyle{definition}

\newtheorem{example}[thm]{Example}

\numberwithin{equation}{section}

\newcommand{\Cay}{\operatorname{Cay}}
\newcommand{\Mn}{\mathcal{M}_n}

\calclayout

\usepackage{hyperref}

\usepackage{xcolor}
\usepackage{soul}

\begin{document}

\title[Cancellative, singly aligned, non-embeddable monoids]{A collection of cancellative, singly aligned, non-embeddable monoids}

\author{Milo Edwardes}
\address{M. Edwardes, Department of Mathematics, University of Manchester, M13 9PL}
\email{milo.edwardes@manchester.ac.uk}

\author{Daniel Heath}
\address{D. Heath, Department of Mathematics, University of Manchester, M13 9PL}
\email{daniel.heath-2@manchester.ac.uk}

\begin{abstract}

By classical results of Malcev, cancellative monoids need not be group-embeddable. In this paper, we describe, give presentations for and study an infinite family $\mathcal{M}_n$ of cancellative monoids which are not group-embeddable, originating from Malcev's work. We show that $\mathcal{M}_n$ is singly aligned for $n \geq 2$, owing to applications in the study of $\mathrm{C}^*$-algebras by Brix, Bruce and Dor-On. We finish by showing that $\mathcal{M}_1$ is not singly aligned, but $2$-aligned.

\end{abstract}

\maketitle

\section{Introduction}

The study of semigroups in their own right emerged in the mid-to-late 1930s motivated by advances in the study of groups and rings \cite{hollings:embedding}. Much of the early work was devoted to the \textit{embedding problem}: given a semigroup $S$, does there exist a group $G$ into which $S$ embeds? More generally: given a category $\mathfrak{A}$, does there exist a faithful functor from $\mathfrak{A}$ to a groupoid $\mathfrak{G}$?

Let $S$ be a semigroup [resp. monoid]. We say $S$ is \textit{group-embeddable} if there exists a group $G$ and a semigroup [resp. monoid] morphism $\phi : S \to G$, such that $\phi$ is injective. We call any $\phi$ with such properties a \textit{semigroup embedding} [resp. \textit{monoid embedding}], and say that $\phi$ \textit{embeds} $S$ in $G$. Otherwise, we say $S$ is \textit{not group-embeddable}, or simply \textit{non-embeddable}.

Many conditions for embeddability would be found over the following decades \cite{bush:MandL,bush:MandL2,clifford:vol2,hollings:embedding,ptak:immerse}, and the embedding problem for categories has also been studied \cite{johnstone:cats}. In both settings, the respective notions of \textit{cancellativity} is seen to be necessary. However, in 1937, Malcev gave an example of a non-embeddable cancellative semigroup, further claiming \textit{``We also have found the necessary and sufficient conditions for the possibility of immersion''} \cite{malcev:original} -- such conditions would appear later \cite{malcev:inf1, malcev:inf2}. 

In more recent developments, the study of $C^*$-algebras has been closely linked to that of semigroups, particularly \textit{cancellative} and \textit{inverse} semigroups \cite{exel:inv,paterson:groupoids}. In \cite[Question B]{brix:coactions}, the following question is posed: given a left cancellative small category $\mathfrak{C}$, is $\mathcal{A}_\lambda(\mathfrak{C})$ canonically $*$-isomorphic to $\partial C^*_\lambda(\mathfrak{C})$? Whilst we do not detail the nuances here, it is noted that the question is resolved positively for groupoid-embeddable categories, and for cancellative \textit{singly aligned} monoids \cite[Theorem 4.17 and Theorem 5.4]{brix:coactions} (therein called \textit{right LCM} monoids). In particular, these monoids need not be group-embeddable, leading to a follow-up question posed by Chris Bruce to the authors of this paper: what are some examples of cancellative, singly aligned monoids which are non-embeddable? Such monoids would truly demonstrate the generalization of the result to categories which are not embeddable into groupoids.

Here, we study an infinite class of monoids $\mathcal{M}_n$ resulting from Malcev's original embeddability conditions to give further such examples. These have presentations naturally arising from a collection of \textit{Malcev Sequences} $\mathcal{I}_n$, with their semigroup counterparts $\mathcal{S}_n$ previously studied in \cite{clifford:vol2,malcev:inf2}.

\textit{Interval monoids} (and ways of constructing them) have been examined by Dehornoy and Wehrung \cite{dehornoy:mr3} -- it is seen that some interval monoids are cancellative and singly aligned, and can be modified slightly to create non-embeddable examples. In particular, a monoid $M_B$ with 24 generators and 11 relations is introduced \cite[Proposition 4.3]{dehornoy:mr3}. Whilst $M_B$ has the same number of generators and relations as what we subsequently define to be $\mathcal{M}_5$, it is not isomorphic. We construct our $\mathcal{M}_n$ monoids via other means -- it remains to be seen if they may be constructed via interval methods in e.g. \cite{dehornoy:mr3, wehrung:gcd}. 

We note that singly aligned monoids encompass the class of \textit{right rigid} monoids \cite{cohn:free}. Results of Doss \cite{doss:immersion} (see also \cite[Theorem 8.12]{cohn:free}) show that cancellative right rigid monoids are always group-embeddable -- consequently the monoids $\mathcal{M}_n$ we introduce are not right rigid.

This paper consists of 4 further sections: firstly, in Section~\ref{sec:prelim}, we provide preliminary definitions regarding notation, \textit{Cayley graphs} and cancellativity. In Section~\ref{sec:mn}, we introduce and study properties of the monoids $\Mn$. In Section~\ref{sec:rLCM}, we show these monoids are singly aligned (for $n \geq 2$), and finally in Section~\ref{sec:2ali}, we discuss the case where $n=1$, showing that whilst $\mathcal{M}_1$ is not singly aligned, it is $2$-aligned.
 
\section{Preliminaries}\label{sec:prelim}

We assume the reader is familiar with general semigroup and monoid theory, particularly ideals, generators, relators and presentations. For comprehensive introductions, we direct the reader to \cite{clifford:vol1,clifford:vol2,howie:fundamentals}. Throughout, we will denote semigroup presentations by $\mathbf{{Sgp}}\langle\ \cdot\ |\ \cdot\ \rangle$ and monoid presentations by $\mathbf{Mon}\langle\ \cdot\ |\ \cdot\ \rangle$. We write $\mathbb{N}$ for the set of positive integers, and $\mathbb{N}_0$ for the set of non-negative integers.

Let $X$ be a non-empty set. We call $X$ an \textit{alphabet} and the elements of $X$ \textit{letters}. A \textit{word} is a (possibly empty) string of letters, i.e. an element of the \textit{free monoid} $X^*$. The number of letters of a word $w$ is its \textit{length}, denoted $|w|$. 

Let $M$ be a semigroup (or monoid) defined by presentation $\langle X\ |\ R \rangle$. We treat a relation in $R$ as an element of $X^* \times X^*$, written $(w_1,w_2)$, or sometimes $w_1 = w_2$. We say two words $w$ and $w'$ in $X^*$ are \textit{equal in $M$}, written $w =_M w'$ or simply $w = w'$, if $w$ and $w'$ represent the same element of $M$ under the natural morphism $X^* \to M$. Equivalently, $w = w'$ if there exist finitely-many applications of the relations in $R$ to the word $w$, called \textit{$R$-transitions}, which transforms the word $w$ into $w'$. To avoid confusion, we instead write $w \equiv w'$ if $w$ and $w'$ are the exact same word in $X^*$. Note that $w \equiv w'$ implies $w = w'$, but not necessarily vice versa.

We denote an $x$-labelled edge of an edge-labelled directed graph from a \textit{source} vertex $u$ to a \textit{target} vertex $v$ by the diagram $u \xrightarrow{x} v$, or simply by $u \rightarrow v$ if the label is irrelevant. The \textit{in-degree} [resp. \textit{out-degree}] of a vertex $v$ is the number of edges with target [resp. source] $v$ (if it is finite). We say a vertex $v$ is \textit{reachable} from $u$ if there exists a (possibly empty) directed sequence of edges from $u$ to $v$. The (\textit{generalised, right}) \textit{Cayley graph} of $M$ with respect to the generating set $X$, denoted $\Cay(M;X)$, is a directed graph, edge-labelled by $X$, with vertex set $V(\Cay(M;X)) := M$, and edge set \[E(\Cay(M;X)) := \left\{p \xrightarrow{x} q \ |\ p,q \in M,\ x \in X \textrm{ such that } px = q\right\}.\]

We say a semigroup (or monoid) $M$ is \textit{left cancellative} if $ca = cb \textrm{ implies } a = b$ for all $a,b,c \in M$. Dually, $M$ is \textit{right cancellative} if $ac = bc \textrm{ implies } a = b$. We say $M$ is \textit{cancellative} if it is both left cancellative and right cancellative. A characterisation of Cayley graphs of cancellative monoids is discussed in \cite{caucal:cayley} -- we will use one important property (Fact 6.1).

\begin{prop}[Caucal \cite{caucal:cayley}]\label{prop:codet}
    Cayley graphs of cancellative monoids are co-deterministic, that is if there exist two edges $u_1 \xrightarrow{x} v$ and $u_2 \xrightarrow{x} v$, with identical labels and targets, then the edges coincide.
\end{prop}

It is quickly seen that if a semigroup (or monoid) is group-embeddable then it must be cancellative. The following result of Malcev \cite{malcev:original} shows cancellativity is insufficient.

\begin{thm}[Malcev \cite{malcev:original}]\label{thm:malcevog}
    Let $\mathcal{S}_1$ be the semigroup defined by presentation \[\mathcal{S}_1 = \mathbf{Sgp}\left\langle a,b,c,d,u,v,x,y\ |\ ax = by,\ au = bv,\  cx = dy \right\rangle.\]Then $\mathcal{S}_1$ is a cancellative semigroup which is not group-embeddable.
\end{thm}

The corresponding monoid \[\mathcal{M}_1 = \mathbf{Mon}\left\langle a,b,c,d,u,v,x,y\ |\ ax = by,\ au = bv,\  cx = dy \right\rangle\]is also be seen to be cancellative but not group-embeddable (see Section~\ref{sec:mn}). This monoid will be of particular interest to us in Section~\ref{sec:2ali}.

The final property of monoids we shall study is that of being \textit{finitely aligned}. We say that a (left cancellative) monoid $M$ is finitely aligned for all $p,q \in M$, either $pM \cap qM = \varnothing$, or \[pM \cap qM = \{r_1,\dots,r_k\}M\]for some finite number of elements $r_i \in M$, i.e. if the intersection of any two principal right ideals is finitely generated. If the number of generating elements for any intersection is globally bounded by some $k$, we say that $M$ is \textit{$k$-aligned}. We will be interested mostly when $k=1$ -- in keeping with recent literature, we use the term \textit{singly aligned} in place of $1$-aligned \cite{lawson:general}, though some texts analogously refer to this as $M$ satisfying \textit{Clifford's condition} \cite{lawson:noncomm,nekrashevych:self} or being \textit{right LCM} \cite{brix:coactions}. Singly aligned is a stronger condition to that of \textit{right ideal Howson} and \textit{finitely aligned} \cite{carson:howson} -- we will discuss the latter in Section~\ref{sec:2ali}.

\section{The Monoids \texorpdfstring{$\mathcal{M}_n$}{Mn}}\label{sec:mn}
In Section~\ref{sec:prelim}, we introduced the semigroup $\mathcal{S}_{1}$ and its associated monoid $\mathcal{M}_{1}$, neither of which are group-embeddable \cite{malcev:original}. In this section, we generalise this definition to describe a countable family of monoids $\mathcal{M}_{n}$, such that for each $n\in\mathbb{N}$, $\mathcal{M}_{n}$ is not cancellative and not group-embeddable. We will define $\mathcal{M}_{n}$ in terms of its presentation.

For any $n \in \mathbb{N}$, define $\mathcal{S}_{n}$ be the semigroup $\mathbf{Sgp}\langle X_n\ |\ \rho_n \rangle$ where:
    $$X_{n}:=\{ a, b, c, d, A_{1}, \dots , A_{n}, B_{1}, \dots , B_{n}, C_{1}, \dots ,C_{n}, D_{1}, \dots, D_{n}\},$$
      \[
    \rho_{n}:=\begin{cases}
        (da,A_{1}C_{1})\\
        (A_{1}D_{1},A_{2}C_{2})\\
        \;\;\;\;\;\;\;\;\vdots \\
        (A_{n-1}D_{n-1},A_{n}C_{n})\\
        (A_{n}D_{n},db)\\
        (cb,B_{n}D_{n})\\
        (B_{n}C_{n},B_{n-1}D_{n-1})\\
        \;\;\;\;\;\;\;\;\vdots\\
        (B_{3}C_{3},B_{2}D_{2})\\
        (B_{2}C_{2},B_{1}D_{1}).
    \end{cases}
    \]We define $\mathcal{M}_{n}$ be the monoid defined by presentation $\mathcal{M}_{n}:= \mathbf{Mon}\langle X_{n}\ |\ \rho_{n}\rangle.$ Note in particular that every relation in $\rho_n$ is of the form $(w_1,w_2)$ where $|w_1| = |w_2| = 2$. 
    
    The semigroups $\mathcal{S}_n$ are of critical importance to the study of non-embeddability. They are cancellative yet not group-embeddable, and occur from a natural construction via a specific collection of \textit{Malcev sequences} $\mathcal{I}_n$ \cite{clifford:vol2}. Whilst we do not refer to the details of their construction here, we note that they have the following astonishing property: given any finite set of equational implications (or \textit{quasi-identities}) holding in all groups, then such implications hold in $\mathcal{S}_n$ for sufficiently large $n$. Full details of this may be found in \cite{clifford:vol2, malcev:inf2}.
    
    We now turn our attention to the monoids $\Mn$. By taking $n = 1$, we obtain exactly the monoid analogue of Malcev's non-embeddable semigroup from Theorem \ref{thm:malcevog} (up to relabelling of generators). Akin to $\mathcal{S}_n$, we now show that each $\Mn$ is cancellative but not group-embeddable. Whilst these follow from Malcev's work, we provide direct proofs. 
    
    To show cancellativity, we observe properties of the relation set $\rho_{n}$. Firstly, we define:
    \begin{align*}
    P_{n}&=\{x\in X_{n}\;|\; x\text{ appears as the left letter of a word in some relation in }\rho_{n}\} \\
    &=\{c,d, A_{1},\dots, A_{n}, B_{1}, \dots, B_{n}\}.\\
    Q_{n}&=\{x\in X_{n}\;|\; x\text{ appears as the right letter of a word in some relation in }\rho_{n}\} \\
    &=\{a,b,C_{1},\dots,C_{n},D_{1},\dots,D_{n}\}.
    \end{align*}Note that $P_n \cap Q_n = \varnothing$. We note some more facts regarding $\rho_n$, $P_n$ and $Q_n$ -- the following are routine to show by considering the effect of $R$-transitions on words. We subsequently use the result without reference.

    \begin{lem}
    Let $n \in \mathbb{N}$ and let $w, w' \in X_{n}^*$ have $w' =_{\Mn} w$. Then:
    \begin{enumerate}
        \item $|w| = |w'|$;
        \item The $i^{\textrm{th}}$ letter of $w$ is in $P_n$ if and only if the $i^{\textrm{th}}$ letter of $w'$ is in $P_n$;
        \item The $i^{\textrm{th}}$ letter of $w$ is in $Q_n$ if and only if the $i^{\textrm{th}}$ letter of $w'$ is in $Q_n$.
    \end{enumerate}
    \end{lem}

    \begin{prop}
    For any $n\in\mathbb{N}$, $\Mn$ is cancellative.
    \end{prop}
    \begin{proof}
    We argue for right cancellativity -- left cancellativity is a dual argument. Let $a,b\in\Mn$ with $a\equiv a_{1}\cdots a_{k}$ and $b\equiv b_{1}\cdots b_{l}$ for $a_{1},\dots,a_{k},b_{1},\dots,b_{l}\in X_{n}$. Note that it is sufficient to show that $ac=bc$ implies $a=b$ for a letter $c\in X_{n}$, as we can add more letters and cancel subsequently, proving the result for general words.
    
    Suppose then that $ac\equiv a_{1}\cdots a_{k}c=b_{1}\cdots b_{l}c\equiv bc$. Hence we can reach $bc$ from a sequence of $\rho_{n}$-transitions from $ac$. Since $|ac| = |bc|$, we have $k=l$ and hence we write $a_{1}\cdots a_{k}c=b_{1}\cdots b_{k}c$.
    
    Let us choose a sequence of $\rho_{n}$-transitions from $ac$ to $bc$ of minimal length. We assume for a contradiction that this does not induce a sequence of transitions from $a$ to $b$, i.e. that it includes some transition involving the final letter $c$. If a $\rho_{n}$-transition in our sequence involves $a_{k}c$, then we must have that there exist letters $a_{k}',c'\in X_n$, with either $(a_{k}'c',a_{k}c)\in\rho_{n}$ or $(a_{k}c,a_{k}'c')\in\rho_{n}$ for $a_{k},a_{k}'\in P_{n}$ and $c,c'\in Q_{n}$. But then, since our sequence of transitions must eventually return $c'$ to $c$, the only way this can happen is by returning $a_{k}'c'$ to $a_{k}c$. Thus removing these two steps from our sequence of transitions will give us a new, shorter, sequence of $\rho_{n}$-transitions from $ac$ to $bc$, contradicting minimality. This means that any sequence of $\rho_{n}$-transitions from $ac$ to $bc$ induces a sequence of $\rho_{n}$-transitions from $a$ to $b$, and hence $a=_{\Mn} b$.
    \end{proof}
\begin{prop}
    For any $n \in \mathbb{N}$, $\Mn$ is not group-embeddable (as a monoid).
\end{prop}

\begin{proof}
    Suppose $\Mn$ is group-embeddable -- we aim for a contradiction. Let $\phi:\Mn \to G$ be a (monoid) embedding of $\Mn$ in a group $G$. We first show that for any $2 \leq k \leq n$, we have $\phi(B_k)\phi(A_k)^{-1} = \phi(B_{k-1})\phi(A_{k-1})^{-1}$. Indeed, using injectivity and the relations in $\rho_n$:\begin{align*}
    \phi(B_k)\phi(A_k)^{-1} &= \phi(B_k)\phi(C_k)\phi(C_k)^{-1}\phi(A_k)^{-1}\\
    &= \phi(B_kC_k)\phi(A_kC_k)^{-1}\\    
    &= \phi(B_{k-1}D_{k-1})\phi(A_{k-1}D_{k-1})^{-1}\\ 
    &= \phi(B_{k-1})\phi(D_{k-1})\phi(D_{k-1})^{-1}\phi(A_{k-1})^{-1}\\
    &= \phi(B_{k-1})\phi(A_{k-1})^{-1}.
    \end{align*}It therefore follows that $\phi(B_n)\phi(A_n)^{-1} = \phi(B_1)\phi(A_1)^{-1}$, and consequently that:\begin{align*}
    \phi(ca) &= \phi(c)\phi(a)\\
    &= \phi(c)\phi(b)\phi(b)^{-1}\phi(d)^{-1}\phi(d)\phi(a)\\
    &= \phi(cb)\phi(db)^{-1}\phi(da)\\
    &= \phi(B_nD_n)\phi(A_nD_n)^{-1}\phi(A_1C_1)\\
    &= \phi(B_n)\phi(D_n)\phi(D_n)^{-1}\phi(A_n)^{-1}\phi(A_1)\phi(C_1)\\
    &= \phi(B_n)\phi(A_n)^{-1}\phi(A_1)\phi(C_1)\\
    &= \phi(B_1)\phi(A_1)^{-1}\phi(A_1)\phi(C_1)\\
    &= \phi(B_1)\phi(C_1)\\
    &= \phi(B_1C_1).
    \end{align*}As $\phi$ is injective, we therefore have $ca =_{\Mn} B_1C_1$. But clearly, no $\rho_n$-transitions can apply to the word $ca$. In particular, there is no sequence of $\rho_n$-transitions from $ca$ to $B_1C_1$ -- a contradiction. Therefore $\Mn$ is not group-embeddable.
\end{proof}

We conclude this section by remarking on some more properties of $\Mn$. We define
\begin{align*}
    L_{n}&=\{\text{Words appearing on the left of relations in } \rho_{n}\}\\
    &=\{cb, da, A_{1}D_{1}, \dots, A_{n}D_{n}, B_{2}C_{2},\dots,B_{n}C_{n}\}.\\
    R_{n}&=\{\text{Words appearing on the right of relations in } \rho_{n}\}\\
    &=\{db, A_1C_1, \dots, A_nC_n, B_1D_1, \dots, B_nD_n\}.
\end{align*}

We note that $L_{n} \cap R_{n} = \varnothing$. By construction, every relation in $\rho_{n}$ has a first component in $L_{n}$ and a second component in $R_{n}$. 

For a fixed $n$, the \emph{left normal form} of a word $w\in X_{n}^*$ is obtained by replacing any two-element subword of $w$ in $R_{n}$ with is corresponding word in $L_{n}$ as dictated by $\rho_{n}$.

\begin{prop}
    Left normal form is indeed a normal form for words in $\Mn$, i.e. for any word $w \in X_n^*$, there is a unique word $w'$ in left normal form such that $w =_{\Mn} w'$.
\end{prop}

\begin{proof}
    Let $w \in X_n^*$. By replacing all two-element subwords of $w$ in $R_n$ by its corresponding word in $L_n$, we may clearly obtain a word $w'$ in left normal form. As such a word was obtained via $\rho_n$-transitions, $w = w'$. It remains to show that each $M$-class contains a unique word in left normal form -- this follows from $P_n \cap Q_n = \varnothing$ and $L_n \cap R_n = \varnothing$.
\end{proof}

\begin{example}\label{ex:1}
    Suppose $n = 2$. The left normal form of $abaC_2dbcA_1B_1D_1$ is \[abaC_2A_2D_2cA_1B_2C_2.\]
\end{example}

Finally, since in each $\Mn$ the trivial word $\epsilon$ evaluates to $1_{\Mn}$, all words evaluating to $1_{\Mn}$ must have also length 0 -- the only such word is $\epsilon$. Hence we have the following result.

\begin{prop}\label{prop:acyclic}
    The Cayley Graph $\Cay(\Mn;X_n)$ is directed acyclic.
\end{prop}
\section{Singly aligned Monoids}\label{sec:rLCM}
Recall that a (left) cancellative monoid $M$ is called \textit{singly aligned} if for all $p,q \in M$, either $pM \cap qM = \varnothing$ or \[pM \cap qM = rM\text{ for some }r \in M.\]This section is devoted to showing that (most of) the monoids $\mathcal{M}_n$ defined in Section~\ref{sec:mn} are singly aligned.

Given $n \in \mathbb{N}$, recall the definitions of $\Mn$, $X_n$, $\rho_n$, $P_n$, $Q_n$, $L_n$ and the concept of left normal form. For any $v \in \Mn$, note that the right ideal $v\Mn$ is exactly the set of vertices reachable from $v$ in the Cayley graph of $\Mn$.

Let $v \in \Mn$ with left normal form $a_1\cdots a_{m-1}a_m$ for $a_1,\dots,a_m \in X_n$. We call $v$ an \textit{intersection base} if $a_{m-1}a_m \in L_n$.

\begin{example}
    Let $n = 2$. Recall the left normal form given in Example \ref{ex:1}: the word $abaC_2dbcA_1B_1D_1$ is an intersection base since $B_2C_2 \in L_2$.
\end{example}

In the sequel, we will see that intersection bases are exactly the generators required for non-trivial intersections of principal right ideals.

\begin{lem}\label{lem:int2}
    Let $v \in \Mn$. The vertex of $\Cay(\Mn;X_{n})$ corresponding to $v$ has in-degree at least $2$ if and only if $v$ is an intersection base.
\end{lem}

\begin{proof}
    Let $v$ be given in left normal form $v \equiv a_1a_2\cdots a_m$ for $a_1,\dots,a_m \in X_n$. We identify $v$ with the corresponding vertex of  $\Cay(\Mn;X_{n})$. 
    
    If $v$ is an intersection base, then the final two letters $a_{m-1}a_m$ of $v$ appear in some relation $(a_{m-1}a_m,rs) \in \rho_n$. Note that $a_{m-1} \neq r$ and $a_m \not\equiv s$, and thus \[a_1a_2\cdots a_{m-1} \neq a_1a_2\cdots a_{m-2}r.\]and there exist distinct edges \[(a_1\cdots a_{m-1}) \xrightarrow{a_m} v \text{ and } (a_1\cdots a_{m-2}r) \xrightarrow{s} v\] in $\Cay(\Mn;X_{n})$. Hence $v$ has in-degree at least $2$.

    Conversely, suppose $v$ has in-degree at least $2$. Let $a \xrightarrow{x} v$ and $b \xrightarrow{y} v$ be distinct edges of $\Cay(\Mn;X_{n})$. Since $\Mn$ is cancellative, $x \not\equiv y$ by Proposition \ref{prop:codet}. Since $ax = by$, either both $x,y \in P_n$ or both $x,y \in Q_n$. If both $x,y \in P_n$, then by considering the left normal form $v$ of $ax$ and $by$, we see that $x \equiv y$ and we arrive at a contradiction. Hence both $x,y \in Q_n$.

    Since $x \not\equiv y$ but $x,y \in Q_n$ and $ax = by$, it must be that there is some relation $(cx,dy)$ or $(dy,cx)$ in $\rho_n$ where $cx$ is the final two letter subword of $ax$. In particular, as $v = ax$, the final two-letter subword of $v$ is in $L_n$ as required.
\end{proof}

\begin{lem}\label{lem:intoint}
    If $v$ is an intersection base and there exists a vertex $u$ and edge $u \xrightarrow{x} v$, then $x \in Q_n$.
\end{lem}

\begin{proof}
    In left normal form, the final letter of $ux = v$ is in $Q_n$.
\end{proof}

\begin{lem}\label{lem:star}
    Let $p,q\in \Mn \setminus \{1\}$ and let $x,y$ be words over $X_{n}$ with lengths at least $1$. Let $x_k$ be the final letter of $x$, and let $y_l$ be the final letter of $y$, and suppose $x_k \not\equiv y_l$. Suppose that $v:= px = qy$ has the following property
    \begin{equation}\label{eq:star}
        \textrm{For any } u \in p\Mn \cap q\Mn\textrm{, there is no edge }u \rightarrow v\textrm{ in }\Cay(\Mn;X_{n}).\tag{$\star$}
    \end{equation}and that there is no path from $p$ to $q$ in $\Cay(\Mn;X_n)$. Then 
    \begin{enumerate}
        \item $px = qy$ is an intersection base;
        \item $x_k,y_l \in Q_n$;
        \item $|x| = 1$, i.e. $x \equiv x_k$.
    \end{enumerate}
    
\end{lem}

\begin{proof}
    Let $v = px = qy$. Write $x \equiv x_1\cdots x_k$ and $y \equiv y_1\cdots y_l$ for letters \linebreak$x_1,\dots,x_k,y_1,\dots,y_l \in X_n$. Note $v \not\in \{p,q\}$.
    
    We claim that the paths in $\Cay(\Mn;X_{n})$ from $p$ to $v$ labelled $x_1,\dots,x_k$ and from $q$ to $v$ labelled $y_1,\dots,y_l$ are vertex-disjoint (except for the vertex $v$). Indeed, if there was some vertex $u$ on both paths, then $u \in p\Mn \cap q\Mn$ and there is a path from $u$ to $v$, contradicting the property \eqref{eq:star} of $v$ (unless $u$ is exactly $v$).
    
    We have $|px|,|qy| \geq 2$. Since the paths labelled $x_1,\dots,x_k$ and $y_1,\dots,y_l$ are vertex-disjoint, we have $x_k \not\equiv y_l$ and $v$ must have in-degree at least $2$. Hence $v$ is an intersection base by Lemma \ref{lem:int2}. Moreover, $x_k, y_l \in Q_n$ by Lemma \ref{lem:intoint}. 
    
    Decompose $qy \equiv q'ay_l$ for a prefix $q' \in X_n^*$ and letter $a \in X_n$. If $k \geq 2$, then since $px = qy$ and $x_k \not\equiv y_l$, we have that $(x_{k-1}x_k,ay_l)$ or $(ay_l,x_{k-1}x_k) \in \rho_n$. Then $px_1\cdots x_{k-2} = q'$ and hence there is a path from $p$ to $q'$ (labelled $x_1\cdots x_{k-2}$). If $|q'| \leq |q|$, then $q'$ is a subword of $q$, and so there exists a path from $q'$ to $q$, but then there exists a path from $p$ to $q$, a contradiction.
    
    If instead $|q'| \geq |q|$, then $q' \in p\Mn \cap q\Mn$, and there exists a path from $q'$ to $v$ (labelled $ay_l$). By the property \eqref{eq:star} of $v$, we must have $v = q'$. But then $q' = v = qy = q'ay_l$, and so $ay_l = 1$ by cancellativity, contradicting that $\Cay(\Mn;X_{n})$ is directed acyclic (Proposition \ref{prop:acyclic}).

    Otherwise, we must have $k = 1$, i.e. $|x| = 1$.
\end{proof}

\begin{prop}\label{prop:cases}
    Let $p,q \in \Mn$ such that $p\Mn \cap q\Mn \neq \varnothing$. Then either:
    \begin{enumerate}
        \item In $\Cay(\Mn;X_{n})$, either $p$ is reachable from $q$ or $q$ is reachable from $p$.
        \item There exist $x,y \in Q_n$ such that $v:= px = qy$ is an intersection base, and for any $u \in p\Mn \cap q\Mn$, there is no edge $u \rightarrow v$ in $\Cay(\Mn;X_{n})$.
    \end{enumerate}
\end{prop}

\begin{proof}
    Suppose $p\Mn \cap q\Mn \neq \varnothing$. Choose $v \in p\Mn \cap q\Mn$ such that there is no $u \in p\Mn \cap q\Mn$ with edge $u \rightarrow v$. Recalling that $\Cay(\Mn;X_{n})$ is directed acyclic (Proposition \ref{prop:acyclic}), note that we may always choose such a $v$ by choosing any element in the intersection and `tracing edges back' in $\Cay(\Mn;X_{n})$ until we can no longer remain in $p\Mn \cap q\Mn$.

    Write $v = px = qy$ for some $x,y \in X_{n}^*$. If $v \in \{p,q\}$, then since there are paths from $p$ to $v$ and $q$ to $v$, we find ourselves in Case (1). Now suppose $v \notin \{p,q\}$, and that there is no path from $p$ to $q$ or $q$ to $p$. In particular, we may assume $p \neq q$ and $p,q,x,y \neq 1$.
    
    We now apply Lemma \ref{lem:star} twice. Since there is no path from $p$ to $q$, Lemma \ref{lem:star} implies that $k = |x| = 1$. Dually, since there is no path from $q$ to $p$, Lemma \ref{lem:star} implies that $l = |y| = 1$. Moreover, $x,y \in Q_n$ and $v$ is an intersection base. We therefore arrive in Case (2).
\end{proof}

We note that Case (1) corresponds exactly to the condition of right rigidity \cite{cohn:free,doss:immersion}, however Case (2) may (and indeed does) arise for elements of general $\mathcal{M}_n$. For example, consider the elements $p = A_1$ and $q = d$.

For the remainder of this section, we now consider only the case where $n \geq 2$. We discuss the $n = 1$ case in Section~\ref{sec:2ali}.

\begin{lem}\label{lem:uniquev}
    Suppose $n \geq 2$. Let $p,q \in \Mn \setminus \{1\}$ and suppose $v_1 := px = qy$ and $v_2:= pw = qz$ for some $x,y,w,z \in X_n$. If $p \neq q$, then $v_1$ and $v_2$ are intersection bases, and $v_1 = v_2$.
\end{lem}

\begin{proof}
    We first show that $v_1$ and $v_2$ are intersection bases. There exist edges $p \xrightarrow{x} v_1$ and $q \xrightarrow{y} v_1$. Since $p \neq q$, we have $x_k \not\equiv y_l$ by Proposition \ref{prop:codet}, and thus $v_1$ must have in-degree at least $2$. Hence $v_1$ is an intersection base by Lemma \ref{lem:int2}. Similar for $v_2$.

     Write $p \equiv p_1\cdots p_k$ and $q \equiv q_1 \dots q_l$ where $p_1,\dots,p_k,q_1,\dots,q_l \in X_n$. By Lemma \ref{lem:intoint}, $x,y,w,z \in Q_n$, and $p_k,q_l \in P_n$. As $p \neq q$, we have $x \not\equiv y$ and $w \not\equiv z$ by cancellativity. In particular, $p_kx \not\equiv q_ly$ and $p_kw \not\equiv q_lz$.
    
    Since $px = qy$ and $pw = qz$, we have that $(p_kx,q_ly),(p_kw,q_lz) \in \rho_n$, up to reordering. Clearly by studying $\rho_n$, we see that we have $p_kx \equiv p_kw$, so $x \equiv w$. In particular, $px = pw$.
\end{proof}

\begin{thm}\label{thm:lcm}
    For any $n \geq 2$, $\Mn$ is singly aligned.
\end{thm}

\begin{proof}
    Let $p,q \in \Mn$ be such that $p\Mn \cap q\Mn \neq \varnothing$. We appeal to Proposition \ref{prop:cases}. Note that if there exists a path in $\Cay(\Mn;X_{n})$ from $p$ to $q$, then $q\Mn \subseteq p\Mn$ and hence $p\Mn \cap q\Mn = q\Mn$. Similarly, if there exists a path from $q$ to $p$, then $p\Mn \cap q\Mn = p\Mn$.

    Otherwise, $p \neq q$, $p \neq 1$ and $q \neq 1$, and moreover Proposition \ref{prop:cases} ensures that there exists some $v \in \Mn$ with the following properties:
    \begin{enumerate}
        \item $v = px = qy$ for some $x,y \in Q_n$, in particular $v \in p\Mn \cap q\Mn$;
        \item $v$ is an intersection base;
        \item For any $u \in p\Mn \cap q\Mn$, there is no edge $u \rightarrow v$ in $\Cay(\Mn;X_{n})$.
    \end{enumerate}Moreover, Lemma \ref{lem:uniquev} ensures $v$ is the unique vertex with these properties. From property (1), $v\Mn \subseteq p\Mn \cap q\Mn$. We show the reverse inclusion.

    Let $w \in p\Mn \cap q\Mn$. Write $w = prw'$ for some $r,w' \in X_n^*$ such that $pr \in p\Mn \cap q\Mn$ and for any $u \in p\Mn \cap q\Mn$, there is no edge $u \rightarrow pr$ in $\Cay(\Mn;X_{n})$. Note that such a vertex certainly exists, and by definition it satisfies properties (3). We also have $r \neq 1$, otherwise $p \in p\Mn \cap q\Mn$ and there is a path from $q$ to $p$. 
    
    We show that $pr$ satisfies properties (1) and (2). Since $pr \in p\Mn \cap q\Mn$, there exists some $s \in X_n^*$ such that $pr = qs$. Similar to $r$, we have $s \neq 1$ as otherwise $q = pr$ and there is a path from $p$ to $q$. We now apply Lemma \ref{lem:star} twice. It follows that $pr$ is an intersection base, with $|r| = |s| = 1$ and $r,s \in Q_n$.    

    Therefore $pr$ satisfies the properties (1), (2) and (3). By Lemma \ref{lem:uniquev}, we must have $pr = v$. Thus $w = vw' \in v\Mn$ and hence $p\Mn \cap q\Mn = v\Mn$.
\end{proof}
\section{Finitely Aligned Monoids}\label{sec:2ali}

We note that the results from Section \ref{sec:rLCM} only apply in the case where $n \geq 2$. In particular, they do not apply for $\mathcal{M}_{1}$: indeed $da = A_1C_1$ and $db = A_1D_1$ are non-equal intersection bases.
However, we may weaken our results when $n = 1$.

Recall that$$X_{1}:=\{ a, b, c, d, A_{1}, B_{1}, C_{1},D_{1}\},$$
      \[\rho_{1}:= \{ (da,A_{1}C_{1}), (A_{1}D_{1},db),(cb,B_{1}D_{1}) \}\]and consequently $P_1 = \{c,d,A_1,B_1\}$ and $Q_1 = \{a,b,C_1,D_1\}$.

We proceed to weaken Lemma \ref{lem:uniquev} and Theorem \ref{thm:lcm} to Lemma \ref{lem:2vs} and Theorem \ref{thm:2ali} respectively for $\mathcal{M}_1$ -- their proofs follow a similar structure to their Section~\ref{sec:rLCM} counterparts.

\begin{lem}\label{lem:2vs}
    Let $p,q \in \mathcal{M}_1 \setminus \{1\}$ and suppose $v_1 := px = qy$, $v_2:= pw = qz$ and $v_3 := pu = qv$ for some $x,y,w,z,u,v \in X_1$. If $p \neq q$, then $v_1$, $v_2$ and $v_3$ are intersection bases, and either $v_1 = v_3$ or $v_2 = v_3$ or $v_{1}=v_{2}$.
\end{lem}

\begin{proof}
    We first show that $v_1$, $v_2$ and $v_3$ are intersection bases. There exist edges $p \xrightarrow{x} v_1$ and $q \xrightarrow{y} v_1$. Since $p \neq q$, we have $x_k \not\equiv y_l$ by Proposition \ref{prop:codet}, and thus $v_1$ must have in-degree at least $2$. Hence $v_1$ is an intersection base by Lemma \ref{lem:int2}. Similar for $v_2$ and $v_3$.

     Write $p \equiv p_1\cdots p_k$ and $q \equiv q_1 \dots q_l$ where $p_1,\dots,p_k,q_1,\dots,q_l \in X_1$. By Lemma \ref{lem:intoint}, $x,y,w,z,u,v \in Q_1$, and $p_k,q_l \in P_1$. As $p \neq q$, we have $x \not\equiv y$, $w \not\equiv z$ and $u \not\equiv v$ by cancellativity. In particular, $p_kx \not\equiv q_ly$, $p_kw \not\equiv q_lz$ and $p_ku \not\equiv q_lv$.
    
    Since $px = qy$, $pw = qz$ and $pu = qv$, we have that \[(p_kx,q_ly),(p_kw,q_lz),(p_ku,q_lv) \in \rho_1\]up to reordering. By studying $\rho_1$, we see that we must have either $p_ku \equiv p_kw$ or $p_kx \equiv p_kw$ or $p_kx\equiv p_ku$. In particular, either $pu = pw$ or $px = pw$ or $px = pu$.
\end{proof}

\begin{thm}\label{thm:2ali}
    The monoid $\mathcal{M}_1$ is $2$-aligned, but not singly aligned.
\end{thm}

\begin{proof}
    Let $p,q \in \mathcal{M}_1$ be such that $p\mathcal{M}_1 \cap q\mathcal{M}_1 \neq \varnothing$. We appeal to Proposition \ref{prop:cases}. If there exists a path in $\Cay(\mathcal{M}_1;X_{1})$ from $p$ to $q$, then $q\mathcal{M}_1 \subseteq p\mathcal{M}_1$ and so we have that $p\mathcal{M}_1 \cap q\mathcal{M}_1 = q\mathcal{M}_1$. If there exists a path from $q$ to $p$, then similarly $p\mathcal{M}_1 \cap q\mathcal{M}_1 = p\mathcal{M}_1$.

    Otherwise, $p \neq q$, $p \neq 1$ and $q \neq 1$, and moreover Proposition \ref{prop:cases} ensures that there exists some $v \in \mathcal{M}_{1}$ with the following properties:
    \begin{enumerate}
        \item $v = px = qy$ for some $x,y \in Q_1$, in particular $v \in p\mathcal{M}_1 \cap q\mathcal{M}_1$;
        \item $v$ is an intersection base;
        \item For any $u \in p\mathcal{M}_1 \cap q\mathcal{M}_1$, there is no edge $u \rightarrow v$ in $\Cay(\mathcal{M}_1;X_{1})$.
    \end{enumerate}Since $v$ satisfies properties (1) and (2), Lemma \ref{lem:2vs} ensures there are at most two choices for $v$. If there is only one such choice, an identical argument to Theorem \ref{thm:lcm} gives us that $p\mathcal{M}_1 \cap q\mathcal{M}_1 = v\mathcal{M}_1$.
    
    Suppose now there are two distinct choices for $v$; label these $v_1$ and $v_2$. From property (1), $\{v_1,v_2\}\mathcal{M}_1 \subseteq p\mathcal{M}_1 \cap q\mathcal{M}_1$. We show the reverse inclusion -- let $w \in p\mathcal{M}_1 \cap q\mathcal{M}_1$.

    Write $w = prw'$ for some $r,w' \in X_1^*$ such that $pr \in p\mathcal{M}_1 \cap q\mathcal{M}_1$ and for every $u \in p\mathcal{M}_1 \cap q\mathcal{M}_1$, there is no edge $u \rightarrow pr$ in $\Cay(\mathcal{M}_1;X_{1})$. Via an identical argument to that in Theorem \ref{thm:lcm}, $pr$ satisfies the properties (1), (2) and (3). By Lemma \ref{lem:2vs}, we must have $pr \in \{v_1,v_2\}$. Hence $w \in \{v_1,v_2\}\mathcal{M}_1$ and so $p\mathcal{M}_1 \cap q\mathcal{M}_1 = \{v_1,v_2\}\mathcal{M}_1$.

    Finally, we note that $\mathcal{M}_1$ is not singly aligned. Indeed, one sees that the ideal \[A_1\mathcal{M}_1 \cap d\mathcal{M}_1 = \{AB,AC\}\mathcal{M}_1\]is not principally generated.
\end{proof}



\section*{Acknowledgements}
The authors would like to thank Chris Bruce for their communications on the subject of this paper, and Mark Kambites and N\'ora Szak\'acs for their draft comments and support in accessing reference materials.

\end{document}